\documentclass[a4paper,12pt]{article}
\usepackage{amsmath,amssymb,amsxtra,amsthm}

\newtheorem{theorem}{Theorem}[section]

\newtheorem{lemma}[theorem]{Lemma}

\newtheorem{proposition}[theorem]{Proposition}
\theoremstyle{definition}

\newtheorem{definition}{Definition}[section]
\newtheorem{remark}[theorem]{Remark}

\newtheorem*{theoremA}{Theorem A}
\newtheorem*{theoremB}{Theorem B}

\makeatletter
    
    \@addtoreset{equation}{section}
  \makeatother

    
\newcommand{\na}{\mathbb{N}}
\newcommand{\re}{\mathbb{R}}

\def\bra#1{\langle {#1} \rangle }
\def\r2n{\re^n\times (\re^n \setminus \{ 0\})}
\def\dbar{\mbox{\setbox0=\hbox{$d$}$d$\kern-.55\wd0\vbox{%
\hrule height.1ex width.75\wd0\kern1.3ex}}}

\date{}
\begin{document}
\title{Remark on characterization of wave front set by
wave packet transform}

\author{Keiichi Kato, Masaharu Kobayashi and Shingo Ito}
\maketitle
\abstract{
In this paper, we give characterization of usual wave front set and wave
front set in $H^s$ in terms of wave packet transform without any
restriction on basic wave packet.}

\section{Introduction}
In this paper, we discuss on characterization of wave front set in terms
of wave packet transform. 
Wave front set is introduced by L. H\"ormander \cite{Hormander-3}, 
which is one of main tools of microlocal analysis in $C^\infty$
category. 
Wave front set of a distribution is a set of singular points of the 
distribution in phase space. 
Such idea to classify the singularities of generalized functions
 ``microlocally'' has been introduced by M. Sato, J. Bros and
 D. Iagolnitzer and L. H\"ormander independently in the 1970s
(Sato-Kawai-Kashiwara \cite{S-K-K-1}, H\"ormander\cite{Hormander-1}, 
\cite{Hormander-2}, Tr\`eves  \cite{Treves-1}).  

The wave packet transform has been introduced by C\'ordoba-Fefferman
 \cite{C-F-1}.  
 Let $f \in {\mathcal S}^\prime({\mathbb R}^n)$ and 
 $\phi \in {\mathcal S}({\mathbb R}^n)\backslash \{0\}$.
Then the wave packet transform $W_\phi f(x,\xi)$ of $f$ with the wave
 packet generated by a function $\phi$ is defined by 
    \begin{equation*}
    W_\phi f(x,\xi) = \int_{{\mathbb R}^n} 
    \overline{ \varphi(y-x)} f(y) e^{-i y \cdot \xi} dy.
    \end{equation*}
We call the above $\phi$  basic wave packet in this paper. 
Wave packet transform is called short time Fourier transform or windowed
 Fourier transform in several literatures. 
We refer to \cite{Grochenig-1} for more details.

G. B. Folland introduced the characterization of wave front set in terms of 
wave packet transform with a 
positive symmetric Schwartz's function as basic wave packet
(\cite[Theorem 3.22]{Folland-1}).
He proposed a question if the same conclusion is still valid
without the restriction of basic wave packet.
T. \=Okaji \cite[Theorem 2.2]{Okaji-1} has given a partial answer 
with a Schwartz's function $\phi$ satisfying $\int x^\alpha \phi dx \ne
0$ for some $\alpha$ as basic wave packet.
Using such basic wave packet, he has also given sufficient condition and
necessity condition which imply that a point belong in $H^s$ wave front set .
In this paper, we give the characterization of $C^\infty$ wave front set
and $H^s$ wave front set with any Schwartz's function 
which is not identically $0$ as basic wave packet.
The following theorems are our main results.

\begin{theorem}
\label{C_inf}
Let $(x_0,\xi_0)\in \r2n$ and $u\in \mathcal{S}'(\re^n)$. 
The following conditions are equivalent.
  \begin{enumerate}
  \item [(i)] $(x_0,\xi_0)\notin WF(u)$
  \item [(ii)]
  There exist $\phi\in \mathcal{S}(\re^n)\setminus{\{0\}}$,
a neighborhood $K$ of $x_0$ and a conic neighborhood $\Gamma$ of
 $\xi_0$ such that  for all $N\in\na$ and $a\ge 1$ there exists
 a constant $C_{N,a}$ satisfying 
    \begin{align*}
    |W_{\phi_\lambda}u(x,\lambda\xi)|\le C_{N,a}\lambda^{-N}
    \end{align*}
 for all $\lambda\ge 1$, $x\in K$ and $\xi\in\Gamma$ with 
 $a^{-1}\le |\xi|\le a$,
    where $\phi_{\lambda}(x)=\lambda^{n/4}\phi (\lambda^{1/2}x)$.
  \item [(iii)]
  There exist a neighborhood $K$ of $x_0$ and a conic neighborhood $\Gamma$ of
 $\xi_0$ such that  for all $N\in\na$ and $a\ge 1$ there exists
 a constant $C_{N,a}$ satisfying 
    \begin{align*}
    |W_{\phi_\lambda}u(x,\lambda\xi)|\le C_{N,a}\lambda^{-N}
    \end{align*}
 for all $\phi\in \mathcal{S}(\re^n)\setminus{\{0\}}$, $\lambda\ge 1$, $x\in K$ and $\xi\in\Gamma$ with 
 $a^{-1}\le |\xi|\le a$,
    where $\phi_{\lambda}(x)=\lambda^{n/4}\phi (\lambda^{1/2}x)$.
  \end{enumerate}
\end{theorem}

\begin{theorem}
\label{characterization}
Let $s\in\re$, $(x_0,\xi_0)\in \r2n$ and $u\in \mathcal{S}'(\re^n)$. 
The following conditions are equivalent.
  \begin{enumerate}
  \item [(i)] $(x_0,\xi_0)\notin WF_{H^s}(u)$
  \item [(ii)]
  There exist $\phi\in \mathcal{S}(\re^n)\setminus{\{0\}}$,
   a neighborhood $K$ of $x_0$ and a neighborhood
   $V$ of $\xi_0$ such that 
      \begin{align}
      \int_1^\infty  \lambda^{n-1+2s}  \int_{V}  \int_{K}
      |W_{\phi_{\lambda}} u(x,\lambda \xi)|^{2} dx d\xi d\lambda <\infty,
      \end{align}
   where $\phi_{\lambda}(x)=\lambda^{n/4}\phi (\lambda^{1/2}x)$.
  \item [(iii)]
  There exist
   a neighborhood $K$ of $x_0$ and a neighborhood
   $V$ of $\xi_0$ such that 
      \begin{align}
      \int_1^\infty  \lambda^{n-1+2s}  \int_{V}  \int_{K}
      |W_{\phi_{\lambda}} u(x,\lambda \xi)|^{2} dx d\xi d\lambda <\infty
      \end{align}
   for all $\phi\in \mathcal{S}(\re^n)\setminus{\{0\}}$, 
    where $\phi_{\lambda}(x)=\lambda^{n/4}\phi (\lambda^{1/2}x)$.
  \end{enumerate}
\end{theorem}

\begin{remark}
P. G\'erard \cite[Proposition 1.1]{Gerard-1} has shown that
(i) and (ii) are equivalent if $\phi(x)$ is a Gaussian function 
$e^{-|x^2|/2}$ (see also J. M. Delort \cite[Theorem 1.2]{Delort-1}).
\end{remark}

Let us begin with a brief review of the history of characterizations of
wave front sets by the wave packet transform.
The $C^\infty$ wave front set is defined as follows.

\begin{definition} 
($C^\infty$ wave front set)
 Let $(x_0,\xi_0) \in \r2n$ and $u \in {\mathcal S}^\prime (\re^n)$. 
 We say that $(x_0,\xi_0)$ is not in the $C^\infty$ wave front set of
 $u$, denoted by $(x_0,\xi_0) \notin WF(u)$, if and only if there exist
 a function $\chi\in C^\infty_0(\re^n)$ with $\chi \equiv 1$ near $x_0$
 and a conic neighborhood $\Gamma$ of $\xi_0$ such that for
 all  $N\in\na$ there exists $C_N>0$ satisfying
    \begin{align} 
    \label{C-inf type}
    |\mathcal{F}[\chi f](\xi)|\le C_N (1+|\xi|)^{-N}
    \end{align}
for $\xi \in \Gamma$.
\end{definition}
Concerning the $C^\infty$ wave front set, following theorem is known.

\begin{theoremA}
(G. B. Folland \cite[Theorem 3.22]{Folland-1} and T. \=Okaji
 \cite[Theorem 2.2]{Okaji-1}). 
Let $(x_0,\xi_0)\in \r2n$ and $u\in \mathcal{S}'(\re^n)$.
Suppose that $\phi\in\mathcal{S}(\re^n)$ satisfies 
 $\int_{\re^n} x^\alpha \phi (x)dx\neq 0$ for some multi-indices
 $\alpha$ and put $\phi_{\lambda}=\lambda^{n/4}\phi (\lambda^{1/2}x)$.
Then, $(x_0,\xi_0)\notin WF(u)$ if and only if there exist a
 neighborhood $K$ of $x_0$ and a conic neighborhood $\Gamma$ of
 $\xi_0$ such that  for all $N\in\na$ and for all $a\ge 1$ there exists
 a constant $C_{N,a}$ satisfying 
    \begin{align*}
    |W_{\phi_\lambda}u(x,\lambda\xi)|\le C_{N,a}\lambda^{-N}
    \end{align*}
 for $\lambda\ge 1$, $x\in K$ and $\xi\in\Gamma$ with 
 $a^{-1}\le |\xi|\le a$.
\end{theoremA}

G. B. Folland \cite{Folland-1} has shown that the conclusion follows if
 $\phi\in\mathcal{S}(\re^n)$ is a positive symmetric Schwartz's function.
In T. \=Okaji \cite{Okaji-1}, the proof of Theorem A is given.
T. \=Okaji \cite{Okaji-1} also discuss the $H^s$ wave front set.
The $H^s$ wave front set is defined as follows.

\begin{definition} 
\label{Sobolev def}
($H^s$ wave front set)
Let  $s\in \re$, $(x_0,\xi_0) \in \r2n$ and 
 $u \in {\mathcal S}^\prime (\re^n)$. 
We say that $(x_0,\xi_0)$ is not in the $H^s$ wave front set of $u$,
 denoted by $(x_0,\xi_0) \notin WF_{H^s}(u)$, if and only if there exist
 a conic neighborhood $\Gamma$ of $\xi_0$ and a function
 $\chi \in C^\infty_0(\re^n )$ with $\chi \equiv 1$ near 
 $x_0$ such that 
    \begin{align}\label{Sobolev type}
    \|\bra{\xi}^s \widehat{\chi u}(\xi)\|_{L^2(\Gamma)}< +\infty .
    \end{align}
\end{definition}

Concerning the $H^s$ wave front set, following theorem is known.

\begin{theoremB}
(T.\=Okaji \cite[Theorem 2.4]{Okaji-1})
Let $s\in\re$, $(x_0,\xi_0)\in \r2n$ and $u\in \mathcal{S}'(\re^n)$. 
Suppose that $\phi \in\mathcal{S}(\re^n)$ satisfies 
 $\int_{\re^n}\phi (x)dx\neq 0$ and put
 $\phi_{\lambda}(x)=\lambda^{n/4}\phi (\lambda^{1/2}x)$. 
If there exist a neighborhood $K$ of $x_0$ and a neighborhood
   $V$ of $\xi_0$ such that 
      \begin{align}
     \label{Okaji-1}
      \int_1^\infty  \lambda^{n-1+2s}  \int_{V}  \int_{K}
      |W_{\phi_{\lambda}} u(x,\lambda \xi)|^{2} dx d\xi d\lambda <\infty
      \end{align}
 then $(x_0,\xi_0)\notin WF_{H^s}(u)$. Conversely, if 
 $(x_0,\xi_0)\notin WF_{H^s}(u)$ then there exist a neighborhood $K$ of
 $x_0$ and a neighborhood $V$ of $\xi_0$ such that 
      \begin{align}
      \label{Okaji-2}
      \int_1^\infty  \lambda^{n-1+2s-\varepsilon}  \int_{V}  \int_{K}
      |W_{\phi_{\lambda}} u(x,\lambda \xi)|^{2} dx d\xi d\lambda <\infty
      \end{align}
 for all $\varepsilon >0$.
\end{theoremB}

\begin{remark}
There is a small gap between \eqref{Okaji-1} and \eqref{Okaji-2}.
In Theorem \ref{characterization}, we remove this gap ($\varepsilon =0$)
and weaken the condition of $\phi$.
\end{remark}

\begin{remark}
S. Pilipovi\'{c}, N. Teofanov and J. Toft\cite{P-T-T-1}, \cite{P-T-T-2} 
 treat wave front set in Fourier-Lebesgue space in terms of 
short time Fourier transform, which is called wave packet transform 
in this paper. 
\end{remark}
This paper is organized as follows. In Section 2, we prepare several
propositions and lemmas.
In Section 3, we prove Theorem \ref{characterization}.
In Section 4, we prove Theorem \ref{C_inf} by the method which is used
in section 3.

\section{Preliminary}
\subsection{Notations}
For $x\in \re^n$ and $r>0$, $B(x,r)$ stands $\{ y\in\re^n ; |y-x|<r\}$.
$\mathcal{F}[f](\xi)=\widehat{f}(\xi)=\int_{\re^n}f(x)e^{-ix\cdot\xi}dx$
is the Fourier transform of $f$.
For a subset $A$ of $\re^n$, we denote the complement of $A$ by $A^c$,  
the set of all interior points of $A$ by $A^\circ$ and  the closure of
$A$ by $\overline{A}$.
$\mbox{\boldmath $1$}_A$
 is the characteristic function of $A$, that is, 
$\mbox{\boldmath $1$}_A=1$ for $x\in A$ and 
$\mbox{\boldmath $1$}_A=0$ for $x\in A^c$.
Throughout this paper, $C$, $C_i$, $C'$ and $C'_i$ ($i=1,2,3,\ldots $) serve
as positive constants, 
if the precise value of which is not needed and $C_{\phi}$ denote 
positive constants depending on $\phi$.\\

\subsection{Key Proposition and Lemmas}
First, we prepare the following proposition and lemma. 
Proposition \ref{2.1} is used in section 4
and from Proposition \ref{iikae_1} to Lemma \ref{remainder}
is used in section 3.
Proposition \ref{2.1}, \ref{iikae_1} and Lemma 
\ref{change cut-off} are easy to prove by the standard
method, but we give the proof for reader's convenience in Appendix.

\begin{proposition}
\label{2.1}
Let $(x_0,\xi_0)\in \r2n$ and $u\in \mathcal{S}'(\re^n)$. 
Assume that $\chi \in C^\infty_0(\re^n)$ satisfies $\chi \equiv 1$
 near $x_0$. 
Then, the following conditions are equivalent.
  \begin{enumerate}
  \item [(i)]
  There exists a conic neighborhood $\Gamma$ of $\xi_0$ such that
   for all $N\in\na$ there exists $C_N>0$ satisfying
      \begin{align*}
      |\mathcal{F}[\chi u](\xi)|\le C_N(1+|\xi|)^{-N}
      \end{align*}
  for $\xi\in\Gamma$.
  \item [(ii)]
  There exists a neighborhood $V$ of $\xi_0$ such that
   for all $N\in\na$ there exists $C_N>0$ satisfying
      \begin{align*}
      | \mathcal{F}[\chi u](\lambda\xi)| \le C_N (1+\lambda
       |\xi|)^{-N}
      \end{align*}
for $\xi\in V$ and $\lambda\ge 1$.
  \end{enumerate}
\end{proposition}

\medskip

\begin{proposition}
\label{iikae_1}
Let $s\in\re$, $(x_0,\xi_0)\in \r2n$ and $u\in \mathcal{S}'(\re^n)$. 
Assume that $\chi \in C^\infty_0(\re^n)$ satisfies $\chi \equiv 1$ 
near $x_0$. Then, the following conditions are equivalent.
  \begin{enumerate}
  \item [(i)]
  There exists a conic neighborhood $\Gamma$ of $\xi_0$ such that
      \begin{align}
      \label{2.2_1}
      \|\bra{\xi}^s\mathcal{F}[\chi u](\xi)\|_{L^2(\Gamma)}<\infty .
      \end{align}
  \item [(ii)]
  There exists a neighborhood $V$ of $\xi_0$ such that 
      \begin{align}
      \label{2.2_2}
      \int_1^\infty  \lambda^{n-1+2s} 
      \| \mathcal{F}[\chi u](\lambda\xi)\|^2_{L^2(V)}d\lambda <\infty.
      \end{align}
  \end{enumerate}
\end{proposition}

\begin{lemma}
\label{change cut-off}
Let $s\in\re$, $(x_0,\xi_0)\in\r2n$, $u\in \mathcal{S}'(\re^n)$ and
 $\chi\in C^\infty_0(\re^n)$ satisfy $\chi\equiv 1$ near $x_0$.
Suppose that there exists $d>0$ such that 
    \begin{align*}
    \int_1^\infty \lambda^{n-1+2s} \| \mathcal{F}[\chi u](\lambda\xi)
    \|^2_{L^2(B(\xi_0,d))}d\lambda <\infty .
    \end{align*}
Then, there exists a neighborhood $V\subset B(\xi_0,d)$ of $\xi_0$ such
 that for all $\zeta \in \mathcal{S}(\re^n)$ there exists constants
 $C>0$ and $C'>0$
 satisfying 
    \begin{multline}
    \label{2.3_1}
    \int_1^\infty \lambda^{n-1+2s} \| \mathcal{F}[\zeta \chi 
    u](\lambda\xi) \|^2_{L^2(V)}d\lambda\\ \le 
    C \int_1^\infty \lambda^{n-1+2s} \| \mathcal{F}[\chi u](\lambda\xi)
    \|^2_{L^2(B(\xi_0,d))}d\lambda +C'.
    \end{multline}
\end{lemma}


\begin{lemma}
\label{structure}
Let $u\in \mathcal{S}'(\re^n)$, $\phi\in\mathcal{S}(\re^n)\setminus\{ 0\}$,
 $(x_0,\xi_0)\in \re^n\times\re^n\setminus\{ 0\}$ and 
 $\chi \in C^\infty_0 (\re^n)$ with $\chi \equiv 1$ near $x_0$.
Suppose that $K$ is a neighborhood of $x_0$ and satisfies 
$\overline{K}\subset \{ x\in\re^n | \chi(x)= 1 \}^\circ$.
Then, there exists $m\in\na$ (which depends on only $u$) such that 
for all $\mu\in\re$ there exists $C_\mu>0$ satisfying
    \begin{align}
     \label{2.3-1}
     \mbox{\boldmath $1$}_K(x) \left| W_{\phi_\lambda}
     [(1-\chi)u] (x,\lambda\xi)\right| \le C_{\mu}\lambda^{-\mu}
     \bra{\xi}^{m}
    \end{align}
 where  $\phi_\lambda (x)=\lambda^{n/4}\phi (\lambda^{1/2}x)$. 
\end{lemma}

\begin{proof}
Since $u\in \mathcal{S}'(\re^n)$, the structure theorem of
 $\mathcal{S'}$ (see, for example \cite[Theorem 2.14]{Mizohata-1})
 yields that there exist $l,m\in\na$ and functions $f_\alpha \in L^2(\re^n)$
 such that 
    $$u(y)=\bra{y}^l \sum_{|\alpha|\le m}\partial^\alpha f_\alpha (y),$$
 where $\alpha$ denotes multi-indices.
Put $g(x,y)=\overline{\phi_\lambda (y-x)} \{ 1-\chi (y)\}\bra{y}^l$ and 
 take $N\in\na$ satisfying $4N\ge 2m+n+1$.
Since 
 $(1-\Delta_y)^N e^{-iy\cdot (\lambda\xi-\eta)}
 =\bra{\lambda\xi-\eta}^{2N}e^{-iy\cdot (\lambda\xi-\eta)}$ 
 for $N\in\na$, we have, by Schwarz's inequality and integration by parts, 
    \begin{align}
     \label{2.3-3}
     \Big| W_{\phi_\lambda}&[(1-\chi)u] (x,\lambda\xi)\Big| \\
     &\le \sum_{|\alpha|\le m}
     \bigg|  \int_{\re^n} g(x,y)\partial^\alpha f_\alpha(y) 
     e^{-i\lambda y\cdot \xi}dy   \bigg|\notag\\
    &\le \sum_{|\alpha|\le m}\int_{\re^n} \Big|
      \mathcal{F}[g(x, \cdot)](\lambda \xi-\eta)
     \eta^{\alpha} \mathcal{F}[f_\alpha] (\eta)\Big| d\eta\notag\\
    &\le \sum_{|\alpha|\le m}\|\mathcal{F}[f_\alpha]\|_{L^2}
     \left( \int_{\re^n}
     \bigg| \eta^\alpha \int_{\re^n} g(x,y)e^{-iy\cdot(\lambda\xi-\eta)}
     dy\bigg|^2 d\eta\right)^{\frac{1}{2}}\notag\\
    &\le \sum_{|\alpha|\le m}\| f_\alpha \|_{L^2} \left\{ \int_{\re^n}
      |\eta|^{2|\alpha|} \left( \int_{\re^n} 
      \dfrac{|(1-\Delta_y)^N g(x,y)|}{\bra{\lambda\xi-\eta}^{2N}}
      dy\right)^2d\eta\right\}^\frac{1}{2}\notag\\
    &\le\sum_{|\alpha|\le m} \|f_\alpha \|_{L^2}   \bigg(   \int_{\re^n}
     \frac{|\eta|^{2|\alpha|}}{\bra{\lambda\xi-\eta}^{4N}}d\eta
     \bigg)^{\frac{1}{2}}   \int_{\re^n} |(1-\Delta_y)^N g(x,y)|dy.
    \end{align}
Since $|\eta| \le |\eta-\lambda\xi|+|\lambda\xi|$, we have
 $\frac{|\eta|}{2}\le|\eta-\lambda\xi|$ or  
 $\frac{|\eta|}{2}\le |\lambda\xi|$.
So, it follows that
    \begin{align}
      \label{2.3-4}
      \int_{\re^n}\frac{|\eta|^{2|\alpha|}}{\bra{\lambda\xi-\eta}^{4N}}
      d\eta
     &\le \int_{\re^n} \bra{\lambda\xi}^{2m}
      \frac{|\eta|^{2m}}{\bra{\lambda\xi}^{2m}
      \bra{\lambda\xi-\eta}^{4N}}d\eta
    \le C_1\lambda^{2m}\bra{\xi}^{2m}.
    \end{align}
On the other hand, 
    \begin{multline}
    \label{2.3-5}
    \int_{\re^n} |(1-\Delta_y)^N g(x,y)|dy\\
    \le   \sum_{|\beta_1|+|\beta_2|+|\beta_3|\le 2N}
     \int_{\re^n} |\partial^{\beta_1}_y \phi_\lambda (y-x)
    \cdot \partial^{\beta_2}_y\{1-\chi (y)\} \cdot \partial^{\beta_3}_y
    \bra{y}^l |dy 
    \end{multline}
 holds.
If $x\in K$ and $y\in {\rm supp}\{\partial^{\beta_2}_y(1-\chi (y))\}$
 then there exist $C_2,C'_2>0$ such that $|y-x|\ge C_2$ and thus
 $|y-x|\ge C'_2\bra{y-x}$, which shows that 
    \begin{align*}
     |\partial^{\beta_3}_y \bra{y}^l|\le C_3\bra{y}^l
     \le C_3'\bra{y-x}^M\bra{x}^l
     =C_M|y-x|^M \bra{x}^l
    \end{align*}
 for $M \ge l$.
Thus we have
    \begin{align}
    \label{2.3-6}
    & \mbox{\boldmath $1$}_K(x) \int_{\re^n} |\partial^{\beta_1}_y
     \phi_\lambda (y-x)
    \cdot \partial^{\beta_2}_y\{1-\chi (y)\} \cdot \partial^{\beta_3}_y
    \bra{y}^l |dy\notag\\
    &\le C'_M\,\mbox{\boldmath $1$}_K(x) \lambda^{\frac{1}{2}(n+|\beta_1|-M)}
    \int_{\re^n}   |
    (\partial^{\beta_1}_y \phi)(\lambda^{\frac{1}{2}}(y-x))|
    \, (\lambda^{\frac{1}{2}}|y-x|)^M\bra{x}^l dy \notag\\
    &\le C''_M \lambda^{N-\frac{M}{2}}.
\end{align}
From \eqref{2.3-3}, \eqref{2.3-4}, \eqref{2.3-5} and \eqref{2.3-6}
we obtain 
    \begin{align}
    \label{2.3-7}
    \mbox{\boldmath $1$}_K(x) \left| W_{\phi_\lambda}
    [(1-\chi)u] (x,\lambda\xi)\right| \le C'''_M\lambda^{m+N-\frac{M}{2}}
    \bra{\xi}^{m}.
    \end{align}
Thus, 
if we take $M$ sufficiently large, we obtain \eqref{2.3-1}.
\end{proof}

\begin{lemma}
\label{hodai 1}
Let $\phi \in \mathcal{S}(\re^n)$, $\lambda \ge1$ and 
 $\delta, k >0$. 
Set $A=\{ \eta \in\re^n ; |\eta|\ge \delta \lambda^{3/4}\}$.
Then, for all $q>0$ there exists a $C>0$ such that 
    \begin{align}
    \int_A \left| \bra{\lambda\eta}^k  \mathcal{F}[\phi]
    (\lambda^{-\frac{1}{2}}\eta)\right|^2 d\eta \le C
    \lambda^{-q}
\end{align}
for all $\lambda \ge 1$.

\begin{proof}
If $\eta\in A$, then 
 $|\lambda^{-\frac{1}{2}}\eta| \ge \delta\lambda^{\frac{1}{4}}$
 and thus
    \begin{align*}
    |\mathcal{F}[\phi](\lambda^{-1/2}\eta)| 
    \le \dfrac{|\lambda^{-\frac{1}{2}}\eta|^p |\mathcal{F}[\phi]
    (\lambda^{-\frac{1}{2}}\eta)|}{\delta^p \lambda^{\frac{p}{4}}}
   \end{align*}
 for all $p>0$.
Since $\lambda \ge1$, simple calculation yields that 
 $\bra{\lambda\eta}^2\le \lambda^3\bra{\lambda^{-\frac{1}{2}}\eta}^2$.
Thus, we have
    \begin{align*}
    \int_A \left|\bra{\lambda\eta}^k  \mathcal{F}[\phi]
    (\lambda^{-\frac{1}{2}}\eta)\right|^2 d\eta
    &\le\int_A\lambda^{3k}\bra{\lambda^{-\frac{1}{2}}\eta}^{2k}
    \dfrac{|\lambda^{-\frac{1}{2}}\eta|^{2p} |\mathcal{F}[\phi]
    (\lambda^{-\frac{1}{2}}\eta)|^2}{\delta^{2p}\lambda^{\frac{p}{2}}}d\eta\\
    &=\lambda^{3k-\frac{p}{2}+\frac{n}{2}} \delta^{-2p}
     \|\bra{\cdot}^{2k} |\cdot|^{2p}\mathcal{F}[\phi]\|_{L^2(\re^n)}
    \end{align*}
 by change of variables.
Therefore, if we take $p$ sufficiently large then 
 we obtain the desired result.
\end{proof}
\end{lemma}

\begin{lemma}
\label{remainder}
Let $k \in\na$, $\chi\in C^\infty_0(\re^n)$,
 $\psi\in \mathcal{S}(\re^n)\setminus \{ 0\}$ and 
 $f\in B^\infty (\re^{2n})$.
Put 
    \begin{equation}
    F_{\alpha,\beta}(\eta,\xi)=\iint_{\re^{2n}}\chi (y)\partial^\alpha_x
    f(x,y)\partial^\beta_x (\psi_\lambda (y-x))e^{-ix\cdot\eta-iy\cdot\xi}
    dxdy
    \end{equation}
 for multi-indices $\alpha, \beta$, where 
 $\psi_\lambda (x)=\lambda^{n/4}\psi (\lambda^{1/2}x)$.
Then,
    \begin{equation}
    \| \bra{\xi'}^k F_{\alpha,\beta}(\eta,\lambda\xi-\xi')
    \|^2_{L^2_{\xi'}(\re^n)}\le C\lambda^{\frac{5n}{2}+6k+|\beta| +3}
    \bra{\xi}^{2k +n+1}.
    \end{equation}
\end{lemma}

\begin{proof}
Let $N=2k+n+1$.
Since $(1-\Delta_y)^N e^{-iy\cdot (\lambda \xi -\xi')}=
 \bra{\lambda\xi-\xi'}^{2N} e^{-iy\cdot (\lambda \xi -\xi')}$,
we have, by integration by parts,
    \begin{align*}
    &\quad  |F_{\alpha,\beta}(\eta,\lambda\xi-\xi')|\\
    &\le \dfrac{1}{\bra{\lambda\xi-\xi'}^{2N}}
     \iint_{\re^{2n}}\Big| (1-\Delta_y)^{N}\{\chi (y)\partial^\alpha_x
     f(x,y)\partial^\beta_x (\psi_\lambda (y-x))\} \Big|dxdy\\
    &\le \dfrac{C}{\bra{\lambda\xi-\xi'}^{2N}}
     \sum_{|\gamma_1|+|\gamma_2|+|\gamma_3|\le 2N}
     \iint_{\re^{2n}}\Big| \partial^{\gamma_1}_y\chi (y)\cdot
     \partial^{\gamma_2}_y\partial^\alpha_x f(x,y) \cdot
     \partial^{\gamma_3}_y\partial^\beta_x (\psi_\lambda (y-x))
     \Big|dxdy\\
    &\le \dfrac{C'}{\bra{\lambda\xi-\xi'}^{2N}} 
     \sum_{|\gamma_1|+|\gamma_3|\le 2N}
     \lambda^{\frac{n}{4}+\frac{|\gamma_3|}{2}+\frac{|\beta|}{2}}
     \iint_{\re^{2n}} | \partial^{\gamma_1}_y\chi (y) \cdot 
     (\partial^{\gamma_3+\beta}\psi) (\lambda^{\frac{1}{2}}(y-x))|dxdy\\
    &= \dfrac{C'}{\bra{\lambda\xi-\xi'}^{2N}} 
     \sum_{|\gamma_1|+|\gamma_3|\le 2N}
     \lambda^{-\frac{n}{4}+\frac{|\gamma_3|}{2}+\frac{|\beta|}{2}}
     \int_{\re^{n}} | \partial^{\gamma_1}_y\chi (y)| dy \int_{\re^n}
     |\partial^{\gamma_3+\beta}\psi (x)|dx\\
    &\le \dfrac{C''}{\bra{\lambda\xi-\xi'}^{2N}} 
     \lambda^{-\frac{n}{4}+N+\frac{|\beta|}{2}}.
    \end{align*}
Thus, it follows that 
    \begin{align*}
    \| \bra{\xi'}^k F_{\alpha,\beta}(\eta,\lambda\xi-\xi')
    \|^2_{L^2_{\xi'}(\re^n)}
    &=\int_{\re^n}\bra{\xi'}^{2k} |F_{\alpha,\beta}
    (\eta,\lambda\xi-\xi')|^2 d\xi'\\
    &\le C \lambda^{-\frac{n}{2}+2N+|\beta|}
     \bra{\lambda\xi}^{N}\int_{\re^n} 
    \dfrac{\bra{\xi'}^{2k}}{\bra{\lambda\xi-\xi'}^{4N}
    \bra{\lambda\xi}^{N}} d\xi'.
    \end{align*}
Since 
$$
\dfrac{\bra{\xi'}^{2k}}{\bra{\lambda\xi-\xi'}^{4N}\bra{\lambda\xi}^{N}}
\le \dfrac{C}{\bra{\xi'}^{n+1}}\quad 
\text{and}\quad \bra{\lambda\xi}^N \le \lambda^N \bra{\xi}^N
\quad \text{for} \quad\lambda \ge 1
$$
 we obtain desired result.
\end{proof}

\section{Proof of Theorem \ref{characterization}}

Trivially, (iii) yields (ii). So, in order to prove Theorem 
 \ref{characterization}, we show that (i) implies (iii) and (ii) implies
 (i). 

\begin{proposition}
\label{1-3} 
Under the same assumption as in  Theorem \ref{characterization},
(i) implies (iii).
\end{proposition}
\begin{proof}
By (i) and Proposition \ref{iikae_1}, there exist a neighborhood $V_1$
 of $\xi_0$ and $\chi\in C^\infty_0(\re^n)$ with $\chi\equiv 1$ near 
 $x_0$ such that 
    \begin{align*}
     \int_1^\infty  \lambda^{n-1+2s}
     \| \mathcal{F}[\chi u] (\lambda \xi)\|^2_{L^2(V_1)}d\lambda 
     <\infty .
    \end{align*}
By Lemma \ref{change cut-off}, there exists $d>0$ such that
 $B(\xi_0,2d)\subset V_1$ and 
    \begin{multline}
    \label{1.2-1}
     \int_1^\infty  \lambda^{n-1+2s}   
     \| \mathcal{F}[\zeta \chi u](\lambda\xi) \|^2_{L^2(B(\xi_0,2d))}d\lambda \\ 
     \le  C \int_1^\infty  \lambda^{n-1+2s} \| \mathcal{F}[\chi u](\lambda\xi)
     \|^2_{L^2(V_1)}d\lambda +C'
    \end{multline}
 for all $\zeta \in \mathcal{S}(\re^n)$.
Let $V=B(\xi_0,d)$ and $K$ be a neighborhood of $x_0$ satisfying 
 $\overline{K}\subset \{x\in\re^n; \chi (x)=1\}^\circ$.
It is enough to show that for all $N\in\na$ there exists $C_N>0$ such that
    \begin{align}
     \label{1.2-2}
      \int_{V} \int_{K} |W_{\phi_{\lambda}} u(x,\lambda \xi)|^{2} dx d\xi
      \le C_N\left(\|\mathcal{F}[\chi u](\lambda \xi)
      \|^2_{L^2(V_1)} +\lambda^{-N}\right).
    \end{align}
In fact, if we take $N=n+[2s]+1$ then 
    \begin{multline*}
    \quad  \int_1^\infty \lambda^{n-1+2s}
      \int_{V} \int_{K} |W_{\phi_{\lambda}} u(x,\lambda \xi)|^{2} dx
     d\xi d\lambda\\
      \le C  \int_1^\infty \lambda^{n-1+2s} \Big(\|\mathcal{F}[\chi u](\lambda \xi)
      \|^2_{L^2(V_1)} +\lambda^{-n-[2s]-1}\Big) d\lambda <\infty .
\end{multline*}
In order to prove \eqref{1.2-2}, we divide 
 $W_{\phi_\lambda}u(x,\lambda\xi)$ into two parts:
    \begin{align*}
     W_{\phi_\lambda}u(x,\lambda\xi) =W_{\phi_\lambda}[\chi^2 u](x,\lambda\xi)
     +W_{\phi_\lambda}[(1-\chi^2)u](x,\lambda\xi).
    \end{align*}
Since $\overline{K}\subset \{ x\in\re^n\,|\, \chi(x)=1\}^\circ$, 
 it follows that for all $N\in\na$ there exists $C_N>0$ such that
    \begin{align}
    \label{1.2-4}
     \int_V \int_K |W_{\phi_\lambda}[(1-\chi^2)u]
     (x,\lambda\xi)|^2dxd\xi \le C_N \lambda^{-N}
    \end{align}
by Lemma \ref{structure}.

Let $\widetilde{\chi}\in C^\infty_0(\re^n)$ satisfy
 $\widetilde{\chi}\equiv 1$ on $K$ and 
 $\overline{{\rm supp}\,\widetilde{\chi}} \subset \{x\in\re^n; \chi (x)=1\}^\circ$.
Applying Taylor's theorem to $\widetilde{\chi} (x)$, we have
     \begin{align}
     \label{1.2-5}
      \widetilde{\chi} (x)
      =\widetilde{\chi}(y)+\sum_{1\le|\alpha|\le L}
      \dfrac{\partial^\alpha_x\widetilde{\chi}(y)}{\alpha !}(x-y)^\alpha
      +\sum_{|\alpha|= L+1}(x-y)^{\alpha}R_\alpha (x,y),
     \end{align}
 where 
    $$R_\alpha (x,y)=\dfrac{L+1}{\alpha !}
    \int_0^1 \partial^\alpha_x\widetilde{\chi}(y+\theta
    (x-y))(1-\theta)^L d\theta.$$
By Plancherel's theorem and \eqref{1.2-5}, we have
    \begin{align*}
     \int_V\int_K |W_{\phi_\lambda}[\chi^2 u](x,\lambda\xi)|^2 dxd\xi
    &\le \int_V\int_{\re^n}  |\widetilde{\chi} (x) 
     W_{\phi_\lambda}[\chi^2 u](x,\lambda\xi)|^2 dxd\xi\notag\\
    &=\int_V  \int_{\re^n} \Big|\mathcal{F}_{x\to\eta}\Big[ 
     \widetilde{\chi} (x) W_{\phi_\lambda}[\chi^2 u](x,\lambda\xi)\Big] 
     (\eta)\Big|^2 d\eta d\xi\notag\\
    &\le C \sum_{j=1}^3 \int_V\int_{\re^n}|I_j|^2 d\eta d\xi,
    \end{align*}
 where
    \begin{align}
    &\label{1.2-7}I_1= \iint_{\re^{2n}} \widetilde{\chi}(y) 
     \overline{\phi_\lambda (y-x)} \{\chi (y)\}^2 u(y) 
     e^{-i\lambda y\cdot\xi-ix\cdot\eta}dydx,\\
    &\label{1.2-8}I_2=\sum_{1\le |\alpha|\le L} \iint_{\re^{2n}} 
     \dfrac{\partial^\alpha_x\widetilde{\chi}(y)}{\alpha !}(x-y)^\alpha
     \overline{\phi_{\lambda}(y-x)}\{\chi (y)\}^2 u(y)
     e^{-i\lambda y\cdot\xi-ix\cdot\eta}dydx
    \end{align}
 and 
    \begin{align}
     \label{1.2-9}
    I_3=\sum_{|\alpha| = L+1}\iint_{\re^{2n}} R_\alpha (x,y)(x-y)^\alpha
     \overline{\phi_\lambda (y-x)}\{\chi (y)\}^2 u(y)
     e^{-i\lambda y\cdot\xi-ix\cdot\eta}dydx.
    \end{align}

First, we consider $I_1$. 
By Fubini's theorem and change of variables, we have
    \begin{align*}
     I_1
     =\lambda^{-\frac{n}{4}} \overline{\mathcal{F}[\phi]
     (\lambda^{-\frac{1}{2}}\eta)}
     \mathcal{F} [\widetilde{\chi}\chi^2 u](\eta+\lambda\xi).
    \end{align*}
Let $B=B(0,d\lambda^{3/4})$.
If $\eta\in B$, $\xi\in V$ and $\eta +\lambda\xi=\lambda\xi'$
 then $\xi'\in B(\xi_0,2d)$. 
So, we have, by \eqref{1.2-1},
    \begin{align}
    \label{1.2-11}
    \int_V\int_{B} |I_1|^2d\eta d\xi
    &=\int_{B} \int_V |\lambda^{-\frac{n}{4}}\mathcal{F}[\phi]
    (\lambda^{-\frac{1}{2}}\eta)|^2 |\mathcal{F}
    [\widetilde{\chi}\chi^2 u](\eta+\lambda\xi)|^2 d\xi d\eta \notag \\
    &\le \int_{B}  |\lambda^{-\frac{n}{4}}\mathcal{F}[\phi]
    (\lambda^{-\frac{1}{2}}\eta)|^2 d\eta 
    \int_{B(\xi_0,2d)}|\mathcal{F}
    [\widetilde{\chi}\chi^2 u](\lambda\xi)|^2 d\xi \notag \\
     &\le \|\phi\|^2_{L^2(\re^n)}\| \mathcal{F}[\chi u](\lambda\xi)
    \|^2_{L^2(V_1)}.
    \end{align}
If $\eta \in B^c$ and $\xi\in V$ then 
 $\bra{\eta +\lambda\xi}\le C\bra{\lambda\eta}$ holds.
Since $u\in \mathcal{S}'(\re^n)$, there exists 
 $k \in\na$ such that $\bra{\xi}^{-k} 
 \mathcal{F}[\psi u](\xi)\in L^2(\re^n)$ for all 
 $\psi\in C^\infty_0(\re^n)$.
Thus, change of variables and Lemma \ref{hodai 1} yeild that
 for all $N\in\na$ there exists $C_N>0$ such that
    \begin{align}
    \label{1.2-10}
    \int_V\int_{B^c} |I_1|^2d\eta d\xi
    &\le C \int_{B^c} \int_V 
    \frac{\bra{\lambda\eta}^{2k}}{\bra{\eta +\lambda\xi}^{2k}}
    \Big|\lambda^{-\frac{n}{4}}\mathcal{F}[\phi]
    (\lambda^{-\frac{1}{2}}\eta) \mathcal{F}
    [\widetilde{\chi}\chi^2 u](\eta+\lambda\xi)\Big|^2 d\xi d\eta \notag \\ 
    & \le C \lambda^{-\frac{3n}{2}}\int_{B^c}
    |\bra{\lambda\eta}^{2k} \mathcal{F}[\phi]
    (\lambda^{-\frac{1}{2}}\eta)|^2 d\eta
     \int_{\re^n}\left|\dfrac{\mathcal{F}
    [\widetilde{\chi}\chi^2 u](\xi)}{\bra{\xi}^k}\right|^2 d\xi\notag\\
    &\le C_N \lambda^{-N}.
    \end{align}

Next, we consider $I_2$. By Fubini's theorem and change of variables, we have
    \begin{align*}
     I_2
     &=\sum_{1\le|\alpha|\le L}
     \dfrac{\lambda^{-\frac{|\alpha|}{2}-\frac{n}{4}}}{\alpha !}
     \mathcal{F}[(-y)^\alpha \phi (y)](\lambda^{-\frac{1}{2}}\eta)
     \mathcal{F}[(\partial^\alpha \widetilde{\chi})\chi^2 u] (\eta +\lambda\xi).
    \end{align*}
Thus, in the similar calculation as above, 
for all $N\in\na$ there exists $C_N,C'_N>0$ such that
    \begin{align}
     \int_V\int_{\re^n}|I_2|^2d\eta d\xi 
     &\label{1.2-16}\le C_N \sum_{1\le |\alpha|\le L}
     \dfrac{\lambda^{-\frac{|\alpha|}{2}}}{\alpha !}
     \left\{ \| \mathcal{F}
     [(\partial^\alpha\widetilde{\chi})\chi^2 u](\lambda\xi) 
     \|^2_{L^2(B(\xi_0,2d))}
     +\lambda^{-N}\right\}\\
     &\label{1.2-12}
     \le C'_N (\| \mathcal{F}[\chi u]
     (\lambda\xi)\|^2_{L^2(V_1)}+\lambda^{-N}).
    \end{align}

Finally, we consider $I_3$. 
Let $M\in\na$ and $4M>n+1$.
Since $(1-\Delta_x)^{M} e^{-ix\cdot \eta}=\bra{\eta}^{2M}e^{-ix\cdot\eta}$,
 we have, by integration by parts and Schwarz's inequality, 
    \begin{align}
    \label{1.2-13}
    |I_3| & \le  \sum_{|\alpha|=L+1}\bigg| \iint_{\re^{2n}}
     \dfrac{(1-\Delta_x)^{M} \{R_\alpha
     (x,y)(x-y)^\alpha \overline{\phi_\lambda (y-x)}\}}{\bra{\eta}^{2M}}
     \{\chi(y)\}^2u(y)e^{-ix\cdot\eta -i\lambda y\cdot \xi}dydx\bigg|\notag\\
    &\le \dfrac{1}{\bra{\eta}^{2M}}\sum_{|\alpha|=L+1}\int_{\re^n}
     |G_\alpha(\eta, \lambda\xi-\xi')\mathcal{F}[\chi u](\xi')
     |d\xi'\notag\\
    &\le \dfrac{1}{\bra{\eta}^{2M}} \left\| 
     \dfrac{\mathcal{F}[\chi u]}{\bra{\cdot}^{k}}
     \right\|_{L^2(\re^n)} \sum_{|\alpha|=L+1}
    \|\bra{\xi'}^{k} G_\alpha(\eta,\lambda\xi-\xi')\|_{L^2_{\xi'}(\re^n)},
    \end{align}
 where
    \begin{equation*}
    G_\alpha (\eta ,\xi)=\iint_{\re^{2n}} \chi (y)
    (1-\Delta_x)^{M}\{R_\alpha
    (x,y)(x-y)^\alpha \overline{\phi_\lambda (y-x)}\}
    e^{-ix\cdot\eta -iy\cdot\xi}dxdy.
    \end{equation*}
Put $g (x)=(-x)^\alpha \overline{\phi (x)}$ and 
 $g_\lambda (x)=\lambda^{\frac{n}{4}}g (\lambda^{\frac{1}{2}}x)$.
Since
    \begin{align*}
    |G_\alpha (\eta,\xi)|
    &=\lambda^{-\frac{|\alpha|}{2}}\bigg| \iint_{\re^{2n}} 
    \chi (y) (1-\Delta_x)^{M}\{R_\alpha
    (x,y)g_\lambda (y-x)\}e^{-ix\cdot\eta -iy\cdot\xi}dxdy\bigg| \\
    &\le  \lambda^{-\frac{L+1}{2}}\sum_{|\beta|+|\gamma|\le 2M}
    \bigg| \iint_{\re^{2n}} \chi(y) \partial^{\beta}_x
    R_\alpha (x,y)\partial^\gamma_x (g_\lambda (y-x))
    e^{-ix\cdot \eta -iy\cdot \xi} dxdy\bigg| ,
    \end{align*}
 we have, by Lemma \ref{remainder},
    \begin{align}
    \label{1.2-14}
    \|\bra{\xi'}^k G_\alpha (\eta ,\lambda\xi-\xi')\|^2_{L^2_{\xi'}(\re^n)}
    \le C\bra{\xi}^{2k+n+1}
    \lambda^{-(L+1)+\frac{5n}{2}+6k+2M+3}.
    \end{align}
Since we can take $L$ arbitrarily, 
\eqref{1.2-13} and \eqref{1.2-14} yeild that
 for all $N\in\na$ there exists $C_N>0$ such that
    \begin{align}
    \label{1.2-15}
    &\int_V\int_{\re^n} |I_3|^2d\eta d\xi\le C_N\lambda^{-N}
    \end{align}
Combining \eqref{1.2-4}, \eqref{1.2-11},  \eqref{1.2-10}, 
 \eqref{1.2-12} and \eqref{1.2-15}, we obtain 
 \eqref{1.2-2}.
\end{proof}

\medskip 
\begin{proposition}
\label{2-1} 
Under the same assumption as in  Theorem \ref{characterization},
(ii) implies (i).
\end{proposition}

\begin{proof}
From the assumption (ii), there exist 
 $\phi\in \mathcal{S}(\re^n)\setminus{\{0\}}$, a neighborhood $K$ of
 $x_0$  and a neighborhood $V_1$ of $\xi_0$ such that 
    \begin{align}
     \label{assump_1}
     \int_1^\infty  \lambda^{n-1+2s}  \int_{V_1}\int_K 
     |W_{\phi_\lambda} u(x,\lambda\xi)|^2 dxd\xi d\lambda <\infty .
    \end{align}
Since $u\in\mathcal{S}'(\re^n)$, there exists $k \in \na$ such that 
 $\|\bra{\xi}^{-k}\mathcal{F}[\psi u](\xi)\|_{L^2(\re^n)}<\infty$
  for all $\psi\in C^\infty_0(\re^n)$.
So, Proposition \ref{iikae_1} yields that there exists a neighborhood
 $V_2$ of $\xi_0$ such that 
    \begin{align}
     \label{assump_2}
     \int_1^\infty  \lambda^{n-1-2k}\| \mathcal{F}[\psi u]
     (\lambda\xi)\|^2_{L^2(V_2)}d\lambda <\infty.
    \end{align}
Let $\chi\in C^\infty_0(\re^n)$ with $\chi\equiv 1$ on $K$
 and let $\widetilde{\chi}\in C^\infty_0(\re^n)$ with
 $\widetilde{\chi} \equiv 1$ near $x_0$,
 $\overline{\mathstrut{\rm supp}\,\widetilde{\chi}}\subset K^\circ$ and 
 $0\le \widetilde{\chi}\le 1$.
Take $d>0$ satisfying $B(\xi_0,d)\subset V_1\cap V_2$.
From Proposition \ref{iikae_1} and \eqref{assump_1}, it is enough to
 show that there exist $\lambda_0\ge 1$ and 
 a neighborhood $V$ of $\xi_0$ such that  
    \begin{align}
    \label{result 2-1}
    \int_{V}|
    \mathcal{F}[\widetilde{\chi}\chi^2 u](\lambda \xi)|^2 d\xi
    \le C\int_{B(\xi_0,d)}\int_K
     |W_{\phi_\lambda}u(x,\lambda\xi)|^2 dxd\xi  +C'\lambda^{-M}
    \end{align}
for $\lambda\ge \lambda_0$ and $M>n+2s$.
Put $V'=B(\xi_0,\frac{d}{2})$.
Since $(A_1+A_2+A_3+A_4)^2\le 4(A^2_1+A^2_2+A^2_3+A^2_4)$ for 
 $A_1,A_2,A_3,A_4\in\re$, we have, by \eqref{1.2-5} and Plancherel's theorem, 
    \begin{align*}
    &\int_{V'}\int_K |W_{\phi_\lambda}u(x,\lambda\xi)|^2dxd\xi\notag\\
    &\quad \ge \int_{V'}\int_{\re^n} |\widetilde{\chi} (x)W_{\phi_\lambda}
    u(x,\lambda\xi)|^2dxd\xi \notag \\
    &\quad =\int_{V'}\int_{\re^n} \Big|
    \mathcal{F}_{x\to\eta}\Big[\widetilde{\chi}(x)
    W_{\phi_\lambda}[\chi^2 u](x,\lambda\xi)+
    \widetilde{\chi}(x)W_{\phi_\lambda}[(1-\chi^2)u](x,\lambda\xi)\Big]
    (\eta)\Big|^2 d\eta d\xi\notag \\
    &\quad =\int_{V'}\int_{\re^n}|I_1+I_2+I_3+I_4|^2d\eta d\xi\notag\\
    &\quad \ge \int_{V'}\int_{\re^n}\left( \dfrac{1}{4}|I_1|^2
     -|I_2|^2 -|I_3|^2 -|I_4|^2\right)d\eta d\xi,
    \end{align*}
 where $I_1, I_2$ and $I_3$ are defined by \eqref{1.2-7}, 
 \eqref{1.2-8} and \eqref{1.2-9} respectively, and
    \begin{align*}
    I_4=\mathcal{F}_{x\to\eta}\big[\widetilde{\chi}(x)W_{\phi_\lambda}
    [(1-\chi^2)u](x,\lambda\xi)\big] (\eta).
    \end{align*}
Thus, 
    \begin{multline}
    \label{I1-4}
    \dfrac{1}{4}\int_{V'}\int_{\re^n}|I_1|^2d\eta d\xi
    \le \int_{V'}\int_K |W_{\phi_\lambda}u(x,\lambda\xi)|^2dxd\xi
    +\int_{V'}\int_{\re^n}|I_2|^2d\eta d\xi\\
    +\int_{V'}\int_{\re^n}|I_3|^2d\eta d\xi
    +\int_{V'}\int_{\re^n}|I_4|^2d\eta d\xi.
    \end{multline}

Let $N\in \na$ with $N>4s+4k$, $\delta=\frac{d}{2(2N-1)}$.
Since
   $ I_1=\mathcal{F}[\phi_\lambda](\eta)
    \mathcal{F}[\widetilde{\chi}\chi^2 u](\eta+\lambda\xi)$,
 we have, by change of variables and Fubini's theorem,
    \begin{align*}
    \int_{V'}\int_{\re^n} |I_1|^2 d\eta d\xi
    &\ge \int_{V'}\int_{B(0,\delta\lambda^{3/4})} |I_1|^2 d\eta d\xi\\
    &=\int_{B(0,\delta\lambda^{3/4})} |\mathcal{F}[\phi_\lambda](\eta)|^2 \int_{V'} 
    |\mathcal{F}[\widetilde{\chi}\chi^2 u](\eta +\lambda\xi)|^2 d\xi
    d\eta \notag\\
    &=\int_{B(0,\delta\lambda^{3/4})}|\mathcal{F}[\phi_\lambda](\eta)|^2 
    \int_{\Omega_{\lambda,\eta}}
    |\mathcal{F}[\widetilde{\chi}\chi^2 u](\lambda\xi)|^2 d\xi
    d\eta \notag\\
    &\ge \int_{B(0,\delta\lambda^{3/4})}|\mathcal{F}[\phi_\lambda](\eta)|^2
    \int_{B(\xi_0,\frac{d}{2}-\delta)}|\mathcal{F}[\widetilde{\chi}
    \chi^2 u](\lambda\xi)|^2 d\xi,
    \end{align*}
 where $\Omega_{\lambda,\eta}=\{ \xi+\frac{\eta}{\lambda}; \xi\in V'\}$.
We note that if $\xi\in V'$, $\eta\in B(0,\delta\lambda^{3/4})$
 and $\lambda\ge1$ then 
 $B(\xi_0,\frac{d}{2}-\delta)\subset \Omega_{\lambda,\eta}$.
Since
    \begin{align*}
    \int_{B(0,\delta\lambda^{3/4})} |\mathcal{F}[\phi_\lambda](\eta)|^2 d\eta
        \longrightarrow \|\phi\|^2_{L^2}\quad\quad (\lambda \rightarrow \infty),
    \end{align*}
there exists $\lambda_0\ge 1$ such that 
    \begin{align*}
    \int_{B(0,\delta\lambda^{3/4})} |\mathcal{F}[\phi_\lambda](\eta)|^2 
    d\eta \ge \dfrac{1}{2}\|\phi\|^2_{L^2}>0
    \end{align*}
 for $\lambda\ge \lambda_0$. 
Therefore we obtain 
    \begin{align}
    \label{2-1 I1}
    \dfrac{1}{2}\|\phi\|^2_{L^2}
    \int_{B(\xi_0,\frac{d}{2}-\delta)}|\mathcal{F}[\widetilde{\chi}
    \chi^2 u](\lambda\xi)|^2 d\xi\le \int_{V'}\int_{\re^n} 
    |I_1|^2 d\eta d\xi    
    \end{align}
 for $\lambda\ge \lambda_0$.
Let $M>n+2s$.
In the similar calculation as \eqref{1.2-16} and \eqref{1.2-4},
 we have
    \begin{align}
    \label{2-1 I2}
    \int_{V'}\int_{\re^n}|I_2|^2d\eta d\xi
    & \le C\bigg( \lambda^{-\frac{1}{2}} \sum_{1\le |\alpha|\le L}
    \| \mathcal{F}
    [(\partial^\alpha\widetilde{\chi})\chi u](\lambda\xi) 
    \|^2_{L^2(B(\xi_0,\frac{d}{2}+\delta))}+
     \lambda^{-M}\bigg)
    \end{align}
 and 
    \begin{align}
    \label{2-1 I4}
    \int_{V'}\int_{\re^n}|I_4|^2 d\eta d\xi \le C\lambda^{-M},
    \end{align}
 respectively.
If we take $L$ sufficiently large, then
in the same way as \eqref{1.2-15} we have
    \begin{align}
    \label{2-1 I3}
    \int_{V'}\int_{\re^n} |I_3|^2d\eta d\xi
     \le C \lambda^{-M}.
     \end{align}
By \eqref{I1-4}, \eqref{2-1 I1}, \eqref{2-1 I2}, \eqref{2-1 I4}
 and \eqref{2-1 I3}, we obtain 
 \begin{multline}
    \label{repeat}
    \int_{B(\xi_0,\frac{d}{2}-\delta)}|\mathcal{F}[\widetilde{\chi}
    \chi^2 u](\lambda\xi)|^2 d\xi 
    \le C\bigg( \int_{B(\xi_0,\frac{d}{2})}
    \int_K|W_{\phi_\lambda}u(x,\lambda\xi)|^2dxd\xi\\
    +\lambda^{-\frac{1}{2}} \sum_{1\le |\alpha|\le L}
    \int_{B(\xi_0,\frac{d}{2}+\delta)} |\mathcal{F}
    [(\partial^\alpha\widetilde{\chi})\chi u](\lambda\xi) 
    |^2 d\xi+\lambda^{-M}\bigg).
 \end{multline}
In the left hand side of above inequality, 
substituting 
$|\mathcal{F}[\widetilde{\chi} \chi^2 u](\lambda\xi)|$
and $B(\xi_0,\frac{d}{2}-\delta)$ for 
$|\mathcal{F}[(\partial^\alpha \widetilde{\chi})\chi^2 u](\lambda\xi)|$
and $B(\xi_0,\frac{d}{2})$, we have
\begin{multline}
    \label{repeat_2}
    \int_{B(\xi_0,\frac{d}{2}+\delta)}|\mathcal{F}[(\partial^\alpha
    \widetilde{\chi})
    \chi^2 u](\lambda\xi)|^2 d\xi 
    \le C\bigg( \int_{B(\xi_0,\frac{d}{2}+2\delta)}
    \int_K|W_{\phi_\lambda}u(x,\lambda\xi)|^2dxd\xi\\
    +\lambda^{-\frac{1}{2}} \sum_{\substack{1\le |\alpha_2|\le L}}
    \int_{B(\xi_0,\frac{d}{2}+3\delta)} |\mathcal{F}
    [(\partial^{\alpha+\alpha_2}\widetilde{\chi})\chi^2 u](\lambda\xi) 
    |^2 d\xi+\lambda^{-M}\bigg).
 \end{multline}
Combining \eqref{repeat} and \eqref{repeat_2}, we have
    \begin{multline}
    \int_{B(\xi_0,\frac{d}{2}-\delta)}|\mathcal{F}[\widetilde{\chi}
    \chi^2 u](\lambda\xi)|^2 d\xi
    \le C\Bigg( \int_{B(\xi_0,\frac{d}{2}+2\delta)}\int_K|W_{\phi_\lambda}
    u(x,\lambda\xi)|^2dxd\xi\\
    \qquad +\lambda^{-1}
    \sum_{\substack{1\le|\alpha_1|\le L\\ 1\le |\alpha_2|\le L}}
    \int_{B(\xi_0,\frac{d}{2}+3\delta)} 
    |\mathcal{F}[(\partial^{\alpha_1+\alpha_2}\widetilde{\chi})\chi^2 u]
    (\lambda\xi)|^2 d\xi+\lambda^{-M}\bigg).
    \end{multline}
Continuing in this fashion, we obtain 
    \begin{align*}
    &\int_{B(\xi_0,\frac{d}{2}-\delta)}|\mathcal{F}[\widetilde{\chi}
    \chi^2 u](\lambda\xi)|^2 d\xi\\
    &\le C \Bigg( \int_{B(\xi_0,\frac{d}{2}+N\delta)}\int_K
    |W_{\phi_\lambda}u(x,\lambda\xi)|^2dxd\xi\\
    &\qquad+\lambda^{-\frac{N}{2}}
    \sum_{\substack{1\le|\alpha_j|\le L\\ 1\le j\le N}}
    \int_{B(\xi_0,\frac{d}{2}+(2N-1)\delta)} 
    |\mathcal{F}[(\partial^{\alpha_1+\cdots +\alpha_N}_x\chi)\chi^2 u]
    (\lambda\xi)|^2 d\xi+\lambda^{-M}\bigg).
    \end{align*}
Since $B(\xi_0,\frac{d}{2}+N\delta)\subset 
 B(\xi_0,\frac{d}{2}+(2N-1)\delta)\subset B(\xi_0,d)$, $M>2n+2s-1$ and 
 $N>4s+4k$, 
 we obtain \eqref{result 2-1} by \eqref{assump_2} as $V=
B(\xi_0,\frac{d}{2}-\delta)$.
\end{proof}

\section{Proof of Theorem \ref{C_inf}}
Trivially, (iii) yields (ii). So, in order to prove Theorem \ref{C_inf}, 
we show that (i) implies (iii) and (ii) implies (i). 

\begin{proposition}
Under the same assumption as in  Theorem \ref{C_inf},
(i) implies (iii).
\end{proposition}
\begin{proof}
By (i), there exists $\chi \in C_0^\infty (\re^n)$ with $\chi\equiv 1$
 near $x_0$ and a conic neighborhood $\Gamma$ of $\xi_0$ such that 
 for all $N\in \na$ there exists $C_N>0$ such that
    \begin{align}
    \label{1.1-1}
     |\mathcal{F}[\chi u](\xi)|\le C_N (1+|\xi|)^{-N}
    \end{align}
 for $\xi\in \Gamma$.
Thus, it follows that for all $N\in\na$ and $a\ge 1$ there exists
 $C_{N,a}>0$ such that
    \begin{align}
     \label{1.1-2}
     |\mathcal{F}[\chi u](\lambda \xi)|\le C_N (1+\lambda|\xi|)^{-N}
     \le C_{N,a}\lambda^{-N}
    \end{align}
 for $\lambda \ge 1$ and $\xi \in \Gamma$ with $a^{-1}\le |\xi|\le a$.

Let $K$ and $K'$ be a neighborhood of $x_0$ satisfying
 $\overline{K}\subset \{ x\in \re^n|\chi (x)=1\}^\circ$ and 
 $\overline{K'}\subset K^\circ$.
Put $V_a=\{ \xi \in \Gamma \,|\, a^{-1}\le |\xi|\le a\}$.
Take $d>0$ satisfying $B(\xi_0,2d)\subset V_a$ and 
 put $V=B(\xi_0,d)$ and $V'=B(\xi_0,\frac{d}{2})$.
Let $\chi_1 \in C^\infty_0(K)$ satisfy $\chi_1 \equiv 1$ on $K'$ and let
 $\chi_2 \in C^\infty_0(V)$ satisfy $\chi_2 \equiv 1$ on $V'$.
Then, by the fundamental theorem of calculus, we have
    \begin{align*}
     &\quad |{\bf 1}_{K'}(x){\bf 1}_{V'}(\xi)W_{\phi_\lambda}u(x,\lambda\xi)|\\
    &\le\left| \int_{-\infty}^{\xi_n} \cdots
             \int_{-\infty}^{\xi_1}\int_{-\infty}^{x_n}\cdots 
             \int_{-\infty}^{x_1}  
      \partial_{\xi_1\cdots \xi_n} \partial_{x_1\cdots x_n}
     \Big\{   \chi_1(x)\chi_2(\xi)W_{\phi_\lambda}u(x,\lambda\xi)  \Big\}
     dx_1\cdots dx_nd\xi_1\cdots d\xi_n\right|\\
     &\le \iint_{\re^{2n}}\Big|\partial_{\xi_1\cdots \xi_n} \partial_{x_1\cdots x_n}
     \big\{   \chi_1(x)\chi_2(\xi)W_{\phi_\lambda}u(x,\lambda\xi)
     \big\}\Big| dxd\xi\\  
     &\le \sum_{0\le \alpha\le \tau}~\sum_{0\le \beta\le \tau}
      C_{\alpha,\beta}\int_{V}\int_{K}
      |  \partial^\alpha_x \partial^\beta_\xi  [W_{\phi_\lambda}u(x,\lambda\xi)]|
     dxd\xi,
    \end{align*}
 where $\tau =(1,1,\ldots, 1)\in\re^n$.
Since
    \begin{align*}
      &\partial_{x_j}(W_{\phi_\lambda}u(x,\lambda\xi))
       =\lambda^{\frac{1}{2}}W_{(\partial_{x_j} \phi)_\lambda}u(x,\lambda\xi),\\
      &\partial_{\xi_j}(W_{\phi_\lambda}u(x,\lambda\xi))
       =-i\lambda W_{\phi_\lambda} [yu](x,\lambda\xi)
    \end{align*}
and Schwarz's inequality, we have
    \begin{multline}
     \label{1.1-3}
      |{\bf 1}_{K'}(x){\bf 1}_{V'}(\xi)W_{\phi_\lambda}u(x,\lambda\xi)|\\
      \le \sum_{0\le \alpha\le \tau}~\sum_{0\le \beta\le \tau}
      C_{\alpha,\beta} \lambda^{\frac{|\alpha|}{2}+|\beta|}
      \left( \int_{V}\int_{K}
      | W_{(\partial^\alpha_x \phi)_\lambda}[y^\beta u](x,\lambda\xi)|^2
      dxd\xi \right)^{\frac{1}{2}}.
    \end{multline}
By the similar way that we got \eqref{1.2-2},
for all $N\in\na$ there exists $C_N>0$ such that
    \begin{align}
    \label{1.1-4}
     \int_{V}\int_{K}
     | W_{(\partial^\alpha_x \phi)_\lambda}[y^\beta u](x,\lambda\xi)|^2
     dxd\xi \le C_N\left(  \int_{V_a}|\mathcal{F}[\chi u](\lambda\xi)|^2 d\xi
     +\lambda^{-N}\right).
    \end{align}
From \eqref{1.1-2}, \eqref{1.1-3} and \eqref{1.1-4}, we obtain the desired result.
\end{proof}

\bigskip

\begin{proposition}
Under the same assumption as in  Theorem \ref{C_inf},
(ii) implies (i).
\end{proposition}

\begin{proof}

From the assumption (ii), there exist 
$\phi\in \mathcal{S}(\re^n)\setminus{\{0\}}$,
a neighborhood $K$ of $x_0$ and a conic neighborhood $\Gamma$ of
 $\xi_0$ such that  for all $N\in\na$ and $a\ge 1$ there exists
 a constant $C_{N,a}$ satisfying 
    \begin{align}
    \label{4.2_1}
    |W_{\phi_\lambda}u(x,\lambda\xi)|\le C_{N,a}\lambda^{-N}
    \end{align}
 for all $\lambda\ge 1$, $x\in K$ and $\xi\in\Gamma$ with 
 $a^{-1}\le |\xi|\le a$.
Take $b> 1$ satisfying that 
$\xi_0\in \{ \xi\in\Gamma \,|\, b^{-1}\le |\xi|\le b \}$.
Put $V_1=\{ \xi\in\Gamma \,|\, b^{-1}\le |\xi|\le b \}$.
By the similar way that we got \eqref{result 2-1},
 there exist $\chi \in C^\infty_0(\re^n)$ with $\chi\equiv 1$ near $x_0$
and a neighborhood $V_2$ of $\xi_0$ 
    \begin{align}
    \label{4.2_2}
    \int_{V_2}|
    \mathcal{F}[\chi u](\lambda \xi)|^2 d\xi
    \le C_N\left( \int_{V_1}\int_K
     |W_{\phi_\lambda}u(x,\lambda\xi)|^2 dxd\xi  +\lambda^{-N}\right)
    \end{align}
for all $N\in\na$.
If $\xi\in V_2$ and $\lambda \ge 1$ then simple calculation yields
 that $\lambda^{-N}\le C_N(1+\lambda |\xi|)^{-N}$.
Thus, from \eqref{4.2_1} and \eqref{4.2_2}, we have
    \begin{align}
    \label{4.2_3}
    \int_{V_2}|
    \mathcal{F}[\chi u](\lambda \xi)|^2 d\xi
    \le C_N(1+\lambda |\xi|)^{-N}
    \end{align}
for $N\in\na$.
Let $V_3$ be a neighborhood of $\xi_0$ satisfying 
$\overline{V_3}\subset V^\circ_2$.
Using Proposition \ref{2.1}, it is enough to show that there
 exists a neighborhood
 $V_3$ of $\xi_0$ such that
    \begin{align}
    \label{4.2_4}
     {\bf 1}_{V_3}(\xi)|\widehat{\chi u}(\lambda\xi)|
     \le C \int_{V_2}| \mathcal{F}[\chi u](\lambda \xi)|^2 d\xi.
    \end{align}

Let $\chi_1 \in C^\infty_0(V_2)$ satisfy $\chi_1 \equiv 1$ on $V_3$.
Then, by the fundamental theorem of calculus, we have
    \begin{align*}
     &\quad |{\bf 1}_{V_3}(\xi)\widehat{\chi u}(\lambda\xi)|\\
     &\le\left| \int_{-\infty}^{\xi_n} \cdots
             \int_{-\infty}^{\xi_1}\partial_{\xi_1\cdots \xi_n}
     \Big\{ \chi_1(\xi)\widehat{\chi u}(\lambda\xi)  \Big\}
     d\xi_1\cdots d\xi_n\right|\\
     &\le \int_{\re^{n}}\Big|\partial_{\xi_1\cdots \xi_n} 
     \big\{  \chi_1(\xi)
     \widehat{\chi u}(\lambda\xi) \big\}\Big| d\xi\\  
     &\le \sum_{0\le \alpha\le \tau}
      C_{\alpha}\int_{V_2}
      | \partial^\alpha_\xi\{\widehat{\chi u}(\lambda\xi)\}|d\xi,
    \end{align*}
 where $\tau =(1,1,\ldots, 1)\in\re^n$.
Since
    \begin{align*}
      &\partial_{\xi_j}\{\widehat{\chi u}(\lambda\xi)\}
       =-i\lambda \mathcal{F}[y\chi u](\lambda\xi)
    \end{align*}
and Schwarz's inequality, we have
    \begin{align}
      |{\bf 1}_{V_3}(\xi)\widehat{\chi u}(\lambda\xi)|
      \le \sum_{0\le \alpha\le \tau}
      C_{\alpha}\lambda^{2|\alpha|} \left( \int_{V_2}
| \mathcal{F}[y\chi u](\lambda\xi)|^2d\xi\right)^{\frac{1}{2}}
    \end{align}
Thus, we obtain the desired result.
\end{proof}

\appendix
\section{appendix}
{\bf Proof of Proposition \ref{2.1}.}
First, under the assumption (i), we show (ii).
From the assumption (i), we have
\begin{align*}
|\mathcal{F}[\chi u](\lambda\xi)|\le C_N(1+\lambda|\xi|)^{-N}
\end{align*}
for $\xi\in\Gamma$.
Let $V$ be a neighborhood of $\xi_0$ satisfying $V\subset \Gamma$.
If $\lambda\ge 1$ and $\xi\in V$ then $\lambda\xi\in \Gamma$.
Thus, we have (ii).

Conversely, we assume (ii).
Put $\Gamma_1=\{ \lambda\xi\,|\, \xi\in V, \lambda >0\}$ and 
$\Gamma_2=\{ \lambda\xi\,|\, \xi\in V, \lambda\ge 1\}$.
By the assumption, we have 
$${\bf 1}_{\Gamma_2}(\xi)|\mathcal{F}[\chi u](\xi)|\le C_N (1+|\xi|)^{-N}.$$

Let $\chi_1 \in C^\infty_0(\re^n)$ satisfy $\chi\equiv 1$ on 
$\Gamma_1\setminus\Gamma_2$.
Then, we have
\begin{align*}
{\bf 1}_{\Gamma_1\setminus\Gamma_2}(\xi)|\mathcal{F}[\chi u](\xi)|
\le |\chi_1(\xi)\mathcal{F}[\chi u](\xi)|
\le C_N (1+|\xi|)^{-N},
\end{align*}
where 
$C_N=\underset{\xi\in \Gamma_1\setminus\Gamma_2}{\rm sup}
|(1+|\xi|)^N \chi_1(\xi)\mathcal{F}[\chi u]|$.
Thus, we obtain the desired result.~\hfill $\square$\\

\bigskip

\noindent{\bf Proof of Proposition \ref{iikae_1}.}
First, we show (i) implies (ii).
Let $\Gamma$ be a conic neighborhood of $\xi_0$ satisfying \eqref{2.2_1}.
Take $\delta >0$ with $|\xi_0|-\delta >0$.
Let $V=\{\xi\in\Gamma ; |\xi_0|-\delta\le |\xi|\le |\xi_0|+\delta\}$,
 $S^{n-1}=\{ \xi\in\re^n ; |\xi|=1\}$ and $A=S^{n-1}\cap \Gamma$.
By change of variables $\xi=r\sigma$, where $r>0$ and $\sigma\in A$, we have
    \begin{align*}
    I&\equiv \int_1^\infty  \lambda^{n-1+2s} 
     \| \mathcal{F}[\chi u](\lambda\xi)\|^2_{L^2(V)}d\lambda\\
    &= \int^\infty_1 \lambda^{n-1+2s} 
     \int_{|\xi_0|-\delta}^{|\xi_0|+\delta}\int_{A}
     r^{n-1}
     |\mathcal{F}[{\chi u}](\lambda r\sigma )|^2 d\sigma
     drd\lambda\\
    & \le C
     \int_{|\xi_0|-\delta}^{|\xi_0|+\delta}
     \int^\infty_1 \int_{A} r (r\lambda)^{n-1+2s}  |\mathcal{F}
     [{\chi u}]( \lambda r\sigma )|^2 d\sigma d\lambda dr.
    \end{align*}
Again by change of variables $\lambda r=\lambda'$ and $r=r'$, we have 
    \begin{align*}
    I&\le C
     \int_{|\xi_0|-\delta}^{|\xi_0|+\delta}
     \int^\infty_{r} \int_{A}  \lambda^{n-1+2s} |\mathcal{F}
     [{\chi u}]( \lambda\sigma)|^2 d\sigma d\lambda dr\\
    & \le 2\delta C
     \int_{|\xi_0|-\delta}^\infty \int_{A}  
     \lambda^{n-1+2s} |\mathcal{F}
     [{\chi u}]( \lambda\sigma)|^2 d\sigma d\lambda \\
    & \le 2\delta C
     \int_0^\infty \int_{A}  \lambda^{n-1}
      (1+\lambda^2|\sigma|^2)^s  |\mathcal{F}
     [{\chi u}]( \lambda\sigma)|^2 d\sigma d\lambda .
    \end{align*}
By change of variables $\xi =\lambda\sigma$, we obtain
    \begin{align*}
     I&\le C \int_\Gamma \bra{\xi}^{2s}
     |\mathcal{F} [{\chi u}]( \xi)|^2 d\xi<\infty .
    \end{align*}
Therefore we obtain \eqref{2.2_2}.

Next, we show (ii) implies (i)
Let $V$ be a neighborhood of $x_0$ satisfying \eqref{2.2_2}.
Take $d>0$ such that $|\xi_0|-d>0$ and $B(\xi_0,d)\subset V$ 
 and put $\Gamma_d=\{ \lambda\xi\,|\, \xi\in B(\xi_0,d),\lambda >0\}$.
Let $\Gamma$ be a conic neighborhood of $\xi_0$ satisfying 
 $\overline{\Gamma}\subset \Gamma_d$.
We divide $\|\bra{\xi}^s\mathcal{F}[\chi u](\xi)\|_{L^2(\Gamma)}$ 
 into two parts:
    \begin{align}
    \label{2.2_3}
    \| \bra{\xi}^s\mathcal{F}[\chi u](\xi)\|^2_{L^2(\Gamma)}
    &=\| \bra{\xi}^s\mathcal{F}[\chi u](\xi)
    \|^2_{L^2(\Gamma\cap \{|\xi|\le 1\})}
    +\| \bra{\xi}^s\mathcal{F}[\chi u](\xi)
    \|^2_{L^2(\Gamma\cap \{|\xi|\ge 1\})} \notag\\
    &\equiv I_1+I_2.
    \end{align}
Since $u\in \mathcal{S}'(\re^n)$, there exists $k\in\na$ such that 
 $\bra{\xi}^{-k}\mathcal{F}[\chi u](\xi)\in L^2(\re^n)$.
Thus, we have
    \begin{align}
    \label{2.2_4}
    I_1&= \int_{\Gamma \cap \{|\xi|\le 1\}}\bra{\xi}^{2s+2k}\bra{\xi}^{-2k}
     |\mathcal{F}[\chi u](\xi)|^2 d\xi\notag\\
    &\le 2^{|s|+k}\int_{\re^n}\bra{\xi}^{-2k}
     |\mathcal{F}[\chi u](\xi)|^2 d\xi \le C.
    \end{align}
Let $A'=\{ \xi \in \Gamma ; |\xi|=1\}$ and take $d'>0$
 satisfying 
 $$V'=\{ \xi\in\Gamma ; |\xi_0|-d'\le |\xi|\le |\xi_0|+d'\} \subset
 B(\xi_0,d).$$ 
By change of variables $\xi=\lambda\sigma$,  we have
    \begin{align*}
    I_2&=\int_1^\infty \int_{A'} \lambda^{n-1}(1+\lambda^2)^s
      |\mathcal{F} [\chi u](\lambda\sigma)|^2 d\sigma d\lambda \\
    &\le C \int_1^\infty \lambda^{n-1+2s}\int_{A'} |
      \mathcal{F} [\chi u](\lambda\sigma)|^2 d\sigma d\lambda \\
    &=\dfrac{C}{2\delta'}\int_{|\xi_0|-d'}^{|\xi_0|+d'}
     \int_1^\infty \lambda^{n-1+2s}\int_{A'} |
      \mathcal{F} [\chi u](\lambda\sigma)|^2 d\sigma d\lambda d\theta.
    \end{align*}
Again by change of variables $\lambda =\lambda'\theta'$ and
 $\theta =\theta'$, we have
    \begin{align*}
    I_2&\le C \int_{|\xi_0|-d'}^{|\xi_0|+d'}
     \int_{\frac{1}{\theta}}^\infty \int_{A'}
     \lambda^{n-1+2s}\theta^{n+2s} |
     \mathcal{F} [\chi u](\lambda\theta\sigma)|^2 
    d\sigma d\lambda d\theta\\
    &\le C'
     \int_{\frac{1}{|\xi_0|+d'}}^\infty \lambda^{n-1+2s}
      \int_{|\xi_0|-d'}^{|\xi_0|+d'}\int_{A'}
      \theta^{n-1} | \mathcal{F} [\chi u]
     (\lambda\theta\sigma)|^2 d\sigma d\theta d\lambda.
    \end{align*}
By change of variables $\xi=\theta\sigma$,
 we have
    \begin{align}
    \label{2.2_5}
    I_2&\le C
    \int_{\frac{1}{|\xi_0|+d'}}^\infty  \lambda^{n-1+2s} 
    \| \mathcal{F}[\chi u](\lambda\xi)\|^2_{L^2(V')}d\lambda \notag\\
    &\le C \bigg(
     \int_1^\infty  \lambda^{n-1+2s} 
    \| \mathcal{F}[\chi u](\lambda\xi)\|^2_{L^2(V')}d\lambda \notag\\
    &\qquad \qquad \qquad \qquad 
    +\bigg| \int_{\frac{1}{|\xi_0|+d'}}^1  \lambda^{n-1+2s} 
    \|\mathcal{F}[\chi u](\lambda\xi)\|^2_{L^2(V')}d\lambda\bigg|\bigg) .
    \end{align}
Since $\bra{\xi}^{-k}\mathcal{F}[\chi u](\xi)\in L^2(\re^n)$, we have 
    \begin{multline}
    \label{2.2_6}
     \bigg| \int_{\frac{1}{|\xi_0|+d'}}^1  \lambda^{n-1+2s} 
    \|\mathcal{F}[\chi u](\lambda\xi)\|^2_{L^2(V')}d\lambda\bigg|\\
    \le C \bigg| \int_{\frac{1}{|\xi_0|+d'}}^1  \lambda^{-1+2s} 
    (1+\lambda^{2})^{k}
    \|\bra{\xi}^{-k}\mathcal{F}[\chi
     u](\xi)\|^2_{L^2(\re^n)}d\lambda\bigg|
    \le C'
    \end{multline}
From \eqref{2.2_3}, \eqref{2.2_4}, \eqref{2.2_5} and \eqref{2.2_6}, 
 we obtain \eqref{2.2_1}.~\hfill $\square$\\

\noindent{\bf Proof of Proposition \ref{change cut-off}.}
Let $\Gamma_d=\{ \lambda\xi\,|\,\xi\in B(\xi_0,d), \lambda>0\}$ and
let $\Gamma_1$ be a conic neighborhood of $\xi_0$ such that 
 $\overline{\Gamma}_1\subset \Gamma_d$.
As shown in the proof of Proposition \ref{iikae_1}, it follows that 
    \begin{align}
    \label{2.3_2}
    \| \bra{\xi}^s\mathcal{F}[\chi u](\xi)\|^2_{L^2(\Gamma_1)}
    \le C \int_1^\infty \lambda^{n-1+2s} \| \mathcal{F}[\chi 
    u](\lambda\xi) \|^2_{L^2(B(\xi_0,d))}d\lambda +C'.
    \end{align}
Let $\Gamma_2$ be a conic neighborhood of $\xi_0$ satisfying 
 $\overline{\Gamma}_2\subset \Gamma_1$.
It is enough to show that 
    \begin{align}
    \label{2.3_3}
    \| \bra{\xi}^s\mathcal{F}[\zeta\chi u](\xi)\|_{L^2(\Gamma_2)}
    \le C\| \bra{\xi}^s\mathcal{F}[\chi u](\xi)\|_{L^2(\Gamma_1)}+C'.
    \end{align}
In fact, Proposition \ref{iikae_1} yields that 
    \begin{align}
    \label{2.3_4}
    \int_1^\infty \lambda^{n-1+2s} \| \mathcal{F}[\zeta \chi 
    u](\lambda\xi) \|^2_{L^2(V)}d\lambda 
    \le C_3 \| \bra{\xi}^s\mathcal{F}[\zeta\chi u](\xi)\|^2_{L^2(\Gamma_2)},
    \end{align}
 where $V=\{ \xi\in\Gamma_2 \,|\, |\xi_0|-\delta\le |\xi|\le
 |\xi_0|+\delta\}$ and $\delta$ is a sufficiently small constant.
Thus, combining \eqref{2.3_2}, \eqref{2.3_3} and \eqref{2.3_4}, 
 we obtain \eqref{2.3_1}.

We divide $\bra{\xi}^s\mathcal{F}[\zeta \chi u](\xi)$ into 
 two parts:
    \begin{align*}
    \bra{\xi}^s\mathcal{F}[\zeta \chi u](\xi)
    &=\bra{\xi}^s \int_{\Gamma_1}\mathcal{F}[\zeta] (\xi-\eta)
    \mathcal{F}[\chi u](\eta )d\eta +\bra{\xi}^s \int_{\Gamma^c_1}
    \mathcal{F}[\zeta] (\xi-\eta) \mathcal{F}[\chi u](\eta )d\eta \\
    &\equiv I_1+I_2.
    \end{align*}
For any $\xi ,\eta \in\re^n$ and $s\in\re$,
 $\bra{\xi}^s\bra{\xi-\eta}^{-|s|}\bra{\eta}^{-s}\le C$ holds.
Thus, Young's inequality yields
    \begin{align*}
    \|I_1\|_{L^2(\Gamma_2)}
    &\le \bigg\|\int_{\re^n} \dfrac{\bra{\xi}^s}{\bra{\xi
    -\eta}^{|s|}\bra{\eta}^s}\bra{\xi -\eta}^{|s|} |\mathcal{F}[\zeta]
     (\xi -\eta )| \mbox{\boldmath $1$}_{\Gamma_1}(\eta)
    \bra{\eta}^s|\mathcal{F}[\chi u](\eta)|d\eta\bigg\|_{L^2(\re^n)}\notag\\
    &\le C \| \bra{\xi}^{|s|}\mathcal{F}[\zeta] (\xi) \|_{L^1(\re^n)}
    \|\bra{\xi}^s \mathcal{F}[\chi u](\xi)\|_{L^2(\Gamma_1)}\notag\\
    &\le C' \|\bra{\xi}^s \mathcal{F}[\chi u](\xi)\|_{L^2(\Gamma_1)}.
    \end{align*}
Since $u\in\mathcal{S}'(\re^n)$, there exists $k\in\na$ such that 
 $\bra{\xi}^{-k}\mathcal{F}[\chi u](\xi)\in L^2(\re^n)$.
If $\eta \notin\Gamma^c_1$ and $\xi\in\Gamma_2$, then 
 $|\xi -\eta|\ge C|\xi|$ and $|\xi -\eta |\ge C|\eta |$.
So we have, by Young's inequality,
    \begin{align*}
    \|I_2\|_{L^2(\Gamma_2)}
    &\le C \bigg\| \bra{\xi}^s  \int_{\Gamma^c_1}
    |\mathcal{F}[\zeta](\xi -\eta)|
    \frac{\bra{\xi-\eta}^{|s|+k}}{\bra{\xi}^{|s|} \bra{\eta}^{k}} 
    |\mathcal{F}[\chi u](\eta )|d\eta  \bigg\|_{L^2(\re^n)} \notag\\
    &\le C' \| \bra{\xi}^{|s|+k}\mathcal{F} [\zeta] (\xi) 
     \|_{L^1(\re^n)} \bigg\|
    \frac{\mathcal{F}[\chi u](\xi)}{\langle\xi\rangle^{k}} 
   \bigg\|_{L^2(\re^n)}\le C''.
    \end{align*}
 Therefore,  we obtain \eqref{2.3_3}.~\hfill $\square$

%
%




\end{document}